\documentclass[12pt,a4paper]{amsart}

\usepackage{amssymb}
\usepackage{amsmath}

\def\ev{{\rm ev}}
\def\C{{\mathbb C}}
\def\Z{{\mathbb Z}}
\def\Q{{\mathbb Q}}

\def\CP{{\mathbb C} {\rm P}}
\def\dim{{\rm dim}}
\def\Res{{\rm Res}}

\def\oM{{\overline {\mathcal M}}}
\def\cL{{\mathcal L}}
\def\cW{\mathcal{W}}

\def\cM{{\mathcal M}}

\def\cL{{\mathcal L}}
\def\cS{{\mathcal S}}
\def\cC{{\mathcal C}}
\def\oC{{\overline C}}
\def\ident{
\begin{picture}(18,10)
\put(3,-1){$\rightarrow$}
\put(3.5,2.5){$\sim$}
\end{picture}
}
\def\rel{{\rm rel}}
\def\ch{{\rm ch}}
\newcommand{\Cs}{\mathcal{C}_\mathrm{s}}

\def\dd{d}
\def\I{\mathrm{I}}
\def\J{\mathrm{J}}

\def\diag{{\rm diag}}

\def\d{{\partial}}

\def\Tr{{\rm Tr}}

\def\p{{\mathbf p}}
\def\la{\left\langle}
\def\ra{\right\rangle}

\title[Equivalence of $r$-ELSV and $r$-BM conjectures]{Equivalence of ELSV and Bouchard-Mari\~{n}o conjectures for $r$-spin Hurwitz numbers}

\date{\today}

\newtheorem{theorem}{Theorem}[section]
\newtheorem{conjecture}[theorem]{Conjecture}
\newtheorem{lemma}[theorem]{Lemma}

\theoremstyle{definition}
\newtheorem{definition}[theorem]{Definition}
\newtheorem{notation}[theorem]{Notation}

\theoremstyle{remark}
\newtheorem{remark}[theorem]{Remark}

\author{S.~Shadrin}

\author{L.~Spitz}

\author{D.~Zvonkine}

\address{S.~S. and L.~S.: Korteweg-de~Vries Institute for Mathematics, University of Amsterdam, P.~O.~Box 94248, 1090 GE Amsterdam, The Netherlands}

\email{S.Shadrin@uva.nl, L.Spitz@uva.nl}

\address{D. Z.: Institut de Mathématiques de Jussieu - Paris Rive Gauche and CNRS 4, place Jussieu Case 247, 75252 PARIS Cedex 05}

\email{zvonkine@math.jussieu.fr}

\begin{document}

\begin{abstract} We propose two conjectures on Huwritz numbers with completed $(r+1)$-cycles, or, equivalently, on certain relative Gro{\-}mov-Witten invariants of the projective line. The conjectures are analogs of the ELSV formula and of the Bouchard-Mari\~no conjecture for ordinary Hurwitz numbers. Our $r$-ELSV formula is an equality between a Hurwitz number and an integral over the space of $r$-spin structures, that is, the space of stable curves with an $r$th root of the canonical bundle. Our $r$-BM conjecture is the statement that $n$-point functions for Hurwitz numbers satisfy the topological recursion associated with the spectral curve $x = -y^r + \log y$ in the sense of Chekhov, Eynard, and Orantin. We show that the $r$-ELSV formula and the $r$-BM conjecture are equivalent to each other and provide some evidence for both.
\end{abstract}

\maketitle

\tableofcontents

\section{Introduction}
In this paper we introduce two conjectures related to Hurwitz numbers with completed cycles that generalize two important theorems in the theory of ordinary Hurwitz numbers. The first is a generalization of the celebrated ELSV formula relating Hurwitz numbers and Hodge integrals~\cite{ELSV} that we call the $r$-ELSV conjecture. The second is a generalization of the (now proved) Bouchard-Mari\~no conjecture~\cite{BM, EMS09, BEMS} relating Hurwitz numbers and the topological recusion procedure of Chekhov, Eynard, and Orantin~\cite{ChE,EO}, that we call the $r$-BM conjecture. We prove that these two conjectures are equivalent, and provide different pieces of evidence for both. 


\subsection{The space of $r$-spin structures}

\subsubsection{The $r$-spin structures on smooth curves}

Fix an integer $r \geq 1$. Let $C$ be a (smooth compact connected) complex curve of genus~$g$ with $n$ marked (pairwise distinct numbered) points $x_1, \dots, x_n$. Denote by $K$ the canonical (cotangent) line bundle over~$C$. Choose $n$ integers $a_1, \dots, a_n \in \{ 0 , \dots, r-1 \}$ such that
$2g-2 - \sum a_i$ is divisible by~$r$.

\begin{definition}
An {\em $r$-spin structure} on~$C$ is a line bundle $\cL$ together
with an identification 
$$
\cL^{\otimes r} \ident K \left( - \sum a_i x_i \right).
$$
The {\em moduli space $\cM^{1/r}_{g,n; a_1, \dots, a_n}$} is the space of all $r$-spin structures on all smooth curves.
\end{definition}

We will denote this space by $\cM^{1/r}$ omitting the other indices if this does not lead to ambiguity.

The element $K( - \sum a_i x_i) \in {\rm Pic}(C)$ can be divided by~$r$ in $r^{2g}$ different ways. Hence there are exactly $r^{2g}$ different $r$-spin structures on every smooth curve $C$. Thus the natural morphism $\pi: \cM^{1/r} \rightarrow \cM_{g,n}$ to the moduli space of smooth curves is actually a nonramified $r^{2g}$-sheeted covering.

The space $\cM^{1/r}$ has a structure of orbifold or smooth Deligne-Mum\-ford stack. The stabilizer of a generic $r$-spin structure is $\Z/r\Z$ (except for several cases with small~$g$ and~$n$ where it can be bigger), because every line bundle $\cL$ has $r$ automorphisms given by the multiplication by $r$th roots of unity. In particular, the degree of $\pi$ is not $r^{2g}$ but $r^{2g-1}$.

\subsubsection{The compactification}
A natural compactification $\oM^{1/r}_{g,n; a_1, \dots, a_n}$
of $\cM^{1/r}_{g,n; a_1, \dots, a_n}$ was constructed 
in~\cite{Jarvis}, \cite{AbrJar}, \cite{CaCaCo},~\cite{Chiodo1}. It has the structure of an orbifold or a smooth Deligne-Mumford stack. There are three different constructions that involve different versions of the universal curve  $\oC^{1/r}_{g,n; a_1, \dots, a_n}$, but the moduli space $\oM^{1/r}_{g,n; a_1, \dots, a_n}$ is the same. 

In the geometrically most natural construction the universal curve over the compactification has orbifold fibers and $\cL$ is extended into a line bundle in the orbifold sense. The second construction is a fiberwise coarsification of the first one. After the coarsification the universal curve becomes a singular orbifold with $A_{r-1}$ singularities and $\cL$ becomes a rank~1 torsion-free sheaf rather than a line bundle. Finally, the third construction is obtained from the second one after resolving its singularities by a sequence of blow-ups. In the third construction the universal curve is smooth and $\cL$ is again a line bundle, but the morphism $\cL^{\otimes r} \to K(-\sum a_i x_i)$ has zeros on the exceptional divisors. The third construction is most convenient for intersection theoretic computations.

On a singular stable curve with an $r$-spin structure, every branch of every node carries an index $a \in \{0, \dots, r-1\}$ such that the sum of indices at each node equals $-2$ modulo~$r$. The divisibility condition $2g-2 - \sum a_i = 0 \bmod r$ is satisfied on each irreducible component of the curve. Thus for a separating node the indices of the branches are determined uniquely, while for a nonseparating node there are $r$ choices.

From now on we write simply $\oM^{1/r}$, and
$\oC^{1/r}$ for the compactified space of $r$-spin structures and its universal curve, if this does not lead to ambiguity.

\begin{remark}
The interest in $\oM^{1/r}$ was initially motivated by Witten's conjecture asserting that some natural intersection numbers on this space can be arranged into a generating series that satisfies the $r$-KdV (or $r$th higher
Gelfand-Dikii) hierarchy of partial differential 
equations~\cite{Witten2}. The conjecture is now proved~\cite{FaShZv} and other beautiful results on the intersection theory of $\oM^{1/r}$
have been obtained~\cite{Chiodo2,ShaZvo}.

One of the main ingredients in Witten's conjecture is the
so-called ``Witten top Chern class'' whose definition uses the sheaves $R^0p_*\cL$ and $R^1p_* \cL$ on $\oM^{1/r}$ in a rather involved way (see~\cite{PolVai} or~\cite{Chiodo3}). The $r$-ELSV conjecture, on the other hand, uses the more straight-forward total Chern class $c(R^1p_*\cL)/c(R^0p_*\cL)$. Therefore it is not at all clear if the $r$-ELSV formula can be related to Witten's conjecture. However, it would provide a link between the intersection theory on $\oM^{1/r}$ and integrable hierarchies via Hurwitz numbers, see Theorem~\ref{Thm:KP}.
\end{remark}

The map $\pi\colon \cM^{1/r} \to \cM_{g,n}$ extends to the compactifications of both moduli spaces $\pi \colon \oM^{1/r} \rightarrow \oM_{g,n}$. Thus the classes $\psi_1, \dots, \psi_n \in H^2(\oM_{g,n}, \Q)$ that one usually defines on $\oM_{g,n}$ (see, for instance,~\cite{Vakil}, Section~3.5) can be lifted to $\oM^{1/r}$. By abuse of notation we will denote the lifted classes by $\psi_i$ instead of $\pi^*(\psi_i)$.

The projection $p:\oC^{1/r} \rightarrow \oM^{1/r}$ induces
a push-forward of $\cL$ in the sense of derived functors:
$$
R^\bullet p_*\cL = R^0 p_* \cL - R^1 p_* \cL.
$$

\begin{notation} \label{Not:R}
We denote by $\cS$ the total Chern class 
$$
\cS = c(-R^\bullet p_*\cL) = c(R^1 p_* \cL)/c(R^0p_* \cL).
$$
\end{notation}

\subsection{The $r$-ELSV conjecture}
\label{SSec:MainConjecture}

Let $\oM^\rel = \oM_{g,m; k_1, \dots, k_n}(\CP^1)$ be the space of stable genus~$g$ maps to $\CP^1$ relative to $\infty \in \CP^1$ with profile $(k_1, \dots, k_n)$ and with $m$ marked points in the source curve. It is a compactification of the space of meromorphic functions on genus~$g$ curves with $n$ numbered poles of orders $k_1, \dots, k_n$ and $m$ more marked points. See~\cite{Vakil}, Section~5 for a precise definition and main properties. Let $\omega \in H^2(\CP^1)$ be the Poincar\'e dual class of a point.

We write every $k_i$ in the form $k_i = r p_i + (r-1-a_i)$, 
where $p_i$ is the quotient and $r-1-a_i$ the remainder of the division of $k_i$ by~$r$. Let $m = (\sum k_i + n + 2g-2)/r$. 

Denote by $h_{g,r;k_1, \dots, k_n}$ the following Gromov-Witten invariants
$$
h_{g,r;k_1, \dots, k_n}
:=
(r!)^m
\!\! 
\int\limits_{\oM^\rel} 
\!\!\!\!
\ev_1^*(\omega) \psi_1^r \cdots \ev_n^*(\omega)\psi_m^r .
$$
Okounkov and Pandharipande~\cite{OkoPan1} studied Gromov-Witten invariants as above and proved that they are equal to so-called {\em Hurwitz numbers with completed cycles} that can be computed combinatorially, see Section~\ref{Sec:CompCyc}.

Introduce the following integrals over the space of $r$-spin structures:
$$
f_{g,r;k_1, \dots, k_n}  =   m! \, r^{m+n+2g-2} \;
\prod_{i=1}^n \frac{\left(\frac{k_i}r\right)^{p_i}}{p_i!} \;\times \!\!\!\!\!\!\!\!
\int\limits_{\oM^{1/r}_{g,n;a_1, \dots, a_n}}
\!\!\!\!\!\!\!\!\!
\frac{\cS}
{\left(1 - \frac{k_1}{r} \psi_1 \right)
\dots
\left(1 - \frac{k_n}{r} \psi_n \right)}.
$$

Chiodo~\cite{Chiodo2} studied the class~$\cS$ expressing it in terms of standard cohomology classes on the moduli space, the coefficients being equal to values of Bernoulli polynomials at rational points with denominator~$r$. His work allows one to compute the numbers $f_{g,r;k_1, \dots, k_n}$.

\begin{conjecture}[$r$-ELSV] \label{Con:main}
We have
$$ 
h_{g,r;k_1, \dots, k_n} = f_{g,r;k_1, \dots,k_n}.
$$
\end{conjecture}

This conjecture suggests a hidden equality between virtual fundamental classes. More precisely, if we consider the space of relative stable maps $f$ such that $df$ has an $r$-th root, the conjecture seems to imply that its virtual fundamental class is obtained from the virtual fundamental class of the space of all stable maps after multiplication by $r$th powers of $\psi$-classes. However this is not elucidated so far and our evidence for this conjecture is not geometric.

The $r$-ELSV formula was first conjectured in~\cite{Zvo}. 
Since both sides are computable, the conjecture can be tested on a computer, and there is virtually no doubt that it is correct. The conjecture is proved in genus~$0$, and also for $(g,n) = (1,1)$ and arbitrary~$r$ (unpublished).

\subsection{The $r$-Bouchard-Mari\~no conjecture}
In~\cite{ChE,EO} Chekhov, Eynard, and Orantin assigned to every plane analytic complex curve $\cC$ a series of invariants $\cW_{g,n}(\cC)$ called {\em correlation $n$-forms}. Each $\cW_{g,n}$ is a meromorphic $n$-form on~$\cC^n$. We do not review the construction here.

\begin{definition}
The {\em $n$-point functions of genus~$g$} for $r$-Hurwitz numbers and for $r$-spin integrals are defined as
\begin{equation}\label{eq:ng-funct1}
H_{g,r}(x_1, \dots, x_n) = \sum_{k_1, \dots, k_n} \!\!\!
\frac{h_{g,r;k_1, \dots, k_n}}{m!} \exp(k_1x_1 + \cdots + k_nx_n).
\end{equation}
\begin{equation}\label{eq:ng-funct2}
F_{g,r}(x_1, \dots, x_n) = \sum_{k_1, \dots, k_n} \!\!\!
\frac{f_{g,r;k_1, \dots, k_n}}{m!} \exp(k_1x_1 + \cdots + k_nx_n).
\end{equation}
\end{definition}

We also let, for any function $f(x_1, \dots, x_n)$, 
$$
Df = \frac{\d^n f}{\d x_1 \cdots \d x_n}
dx_1 \cdots dx_n.
$$

\begin{conjecture}[$r$-BM] \label{Conj:BM}
Let $\mathcal{W}_{g,n}$ be the $n$-point correlation forms of the plane curve $x = -y^r + \log y$. Then we have
\begin{equation}
DH_{g,r}(x_1, \ldots, x_n) = \mathcal{W}_{g,n} .
\end{equation}
\end{conjecture}

\subsection{Main theorem}

\begin{theorem} \label{Thm:Main}
Let $\mathcal{W}_{g,n}$ be the $n$-point correlation forms of the plane curve $x = -y^r + \log y$. Then we have
\begin{equation}
DF_{g,r}(x_1, \ldots, x_n) = \mathcal{W}_{g,n} .
\end{equation}
\end{theorem}

Thus our main theorem shows that the $r$-ELSV conjecture is equivalent to $r$-BM conjecture. The work~\cite{BEMS} provides a ``physics proof" of the Bouchard-Mari\~{n}o conjecture for ordinary Hurwitz numbers, independent of the ELSV-formula. In Section~\ref{Sec:Evidence} we generalize it to our case. The generalization goes through perfectly; however so far we are unable to transform it into a rigorous mathematical proof, though it will probably convince a mathematical physicist.

\subsection{Plan of the paper}
In Section~\ref{Sec:CompCyc} we introduce the completed cycles and completed Hurwitz numbers involved in the $r$-ELSV conjecture. 

In Section~\ref{Sec:Chiodo} we review the work of Chiodo showing that the cohomology classes~$\mathcal{S}$ form a cohomological field theory. We determine that cohomological field theory and the Givental $R$-matrix assigned to~$\cS$.

Section~\ref{sec:SCHur} contains the proof of the main theorem. It is based on the identification~\cite{DOSS12} of the Givental-Teleman theory of semi-simple cohomological field theories with the Chekhov-Eynard-Orantin topological recursion theory. We check that the $R$-matrix for the class~$\cS$ coincides with the $R$-matrix assigned to the curve $x=-y^r+\log y$.

The last two sections are devoted to evidence for the $r$-BM conjecture. We first prove in Section~\ref{sec:HurMM} that completed Hurwitz numbers can be described using a specific matrix model. In Section~\ref{Sec:Evidence} we provide evidence relating this matrix model to the spectral curve of Conjecture~\ref{Conj:BM}. We also provide some more direct evidence for the conjecture based on the ideas of~\cite{DMSS12}.

\subsection{Notation}

The following notation is used consistently throughout this paper.

\begin{itemize}

\item $r \geq 1$ is an integer;
we work with the space of $r$-spin structures.

\item $\I$ is the imaginary unit, $\I^2 = -1$.

\item $\J$ is the primitive $r$th root of unity, $\J^r = 1$.

\item
$g \geq 0$ is the (arithmetic) genus
of the curves~$C$ under consideration.

\item
$n \geq 1$ is the number of marked points on $C$;
the points themselves are denoted by $x_1, \dots, x_n$.

\item
$k_1, \dots, k_n$ is an $n$-tuple of positive integers such that
$m = (\sum k_i + n + 2g-2)/r$ is also an integer. We consider
meromorphic functions on~$C$ with $n$ poles of orders
$k_1, \dots, k_n$ at the marked points $x_1, \dots, x_n$.
Such a function represents a ramified covering of $\CP^1$ 
of degree $\sum k_i$ with a ramification point of type
$(k_1, \dots, k_n)$ at $\infty$. We also require the
covering to have exactly $m$ additional
ramification points of multiplicity~$r$.

\item
Each $k_i$ determines a unique couple of integers 
$p_i, a_i$ such that $k_i = r p_i + (r-1-a_i)$ and 
$a_i \in \{0 , \dots, r-1 \}$.
Thus $p_i$ is the quotient and $a_i$ the 
``reversed remainder'' of the division of $k_i$ by~$r$.
The $r$-spin structures we consider will be $r$th roots of
the line bundle $K(-\sum a_i x_i)$.

\item
When integrating against monomials in $\psi$-classes we will denote their powers by $\psi_1^{d_1} \cdots \psi_n^{d_n}$.
\end{itemize}

The integers $g$, $n$, $r$, $k_1, \dots, k_n$ determine the integers $m$, $p_i$, and $a_i$ uniquely.

\subsection{Acknowledgements} 
L.~S. and S.~S. were supported by a VIDI grant of the Netherlands Organization for Scientific Research. D.~Z. was supported by the grant ANR-09-JCJC-0104-01. 

The $r$-ELSV formula was conjectured by D.~Z.~\cite{Zvo} who had important discussions about it with many people before this paper appeared. We are particularly greatful to Alessandro Chiodo, Maxim Kazarian, Sergey Lando, Sergei Natanzon, Christian Okonek, Rahul Pandharipande, Dmitri Panov, Mathieu Romagny, Mikhail Shapiro, and Claire Voisin. We also thank Ga\"etan Borot, Leonid Chekhov, Petr Dunin-Barkowski, Bertrand Eynard, Motohico Mulase, and Nicolas Orantin for plenty of helpful discussions on closely related topics and teaching us the methods of spectral curve topological recursion.

\section{Completed cycles and completed Hurwitz numbers} 
\label{Sec:CompCyc}
Okounkov and Panharipande proved in~\cite{OkoPan1} that 
the left-hand side of Conjecture~\ref{Con:main} is equal to a certain type of Hurwitz numbers (enumerating ramified coverings of the sphere of a certain kind) involving so-called ``completed cycles''. We introduce them here.

\subsection{Completed cycles}

A {\em partition}~$\lambda$ of an integer~$q$ is a non-increas\-ing
finite sequence $\lambda_1 \geq \dots \geq \lambda_l$ 
such that $\sum \lambda_i =q$. 

It is known that the irreducible representations $\rho_\lambda$
of the symmetric group $S_q$ are in a natural one-to-one
correspondence with the partitions~$\lambda$ of~$q$. 
On the other hand, to a partition~$\lambda$
of~$q$ one can assign a central element $C_{p,\lambda}$ 
of the group algebra $\C S_p$ for any positive integer~$p$. 
The coefficient of a given permutation $\sigma \in S_p$ in
$C_{p,\lambda}$ is defined as the number of ways to choose
and number $l$ cycles of $\sigma$ so that their lengths are 
$\lambda_1, \dots, \lambda_l$, and the remaining $p-q$ elements are fixed points of~$\sigma$. Thus the coefficient of $\sigma$ vanishes unless its cycle lengths are $\lambda_1, \dots, \lambda_l, 1, \dots, 1$. In particular, $C_{p, \lambda} = 0$ if $p < q$. Thus $C_{p, \lambda}$ is
{\em the sum of permutations with $l$ numbered cycles
of lengths $\lambda_1, \dots, \lambda_l$ and any number of non-numbered fixed points}.

The collection of elements $C_{p, \lambda}$ for $p = 1, 2, \dots$
is called a {\em stable center element}\footnote{This has nothing to do with stable curves.}~$C_\lambda$. 
For example, the stable element $C_2$ is
the sum of all transpositions in $\C S_p$, which is well-defined for each~$p$, (in particular, equal to zero for $p=1$).

Let~$\lambda$ be a partition of~$q$ and~$\mu$ a partition of~$p$. Since $C_{p,\lambda}$ lies in the center of $\C S_p$, 
it is represented by a scalar (multiplication by a constant) 
in the representation~$\rho_\mu$ of~$S_p$. Denote this scalar
by $f_\lambda(\mu)$. Thus to a stable center element $C_\lambda$ we have assigned a function $f_\lambda$
defined on the set of all partitions. We are interested in the
vector space spanned by the functions~$f_\lambda$.

To study this space, one defines some new functions 
on the set of partitions as follows\footnote{The standard definition
involves certain additive constants that we have dropped to
simplify the expression, since these constants play no role here.}:
$$
\p_{r+1}(\mu) = \frac1{r+1} \sum_{i \geq 1} 
\left[(\mu_i - i + \frac12)^{r+1} - (-i + \frac12)^{r+1} \right]
\qquad (r \geq 0).
$$

\begin{theorem}[Kerov, Olshansky \cite{KerOls}]
The vector space spanned by the functions $f_\lambda$
coincides with the algebra generated by the functions $\p_1, \p_2, \dots$.
\end{theorem}

As a corollary, to each stable center element
$C_\lambda$ we can assign a polynomial in $\p_1, \p_2, \dots$
and, conversely, each $\p_{r+1}$ corresponds to a linear
combination of stable center elements $C_\lambda$.

\begin{definition} \label{Def:CompCyc}
The linear combination of stable center elements corresponding
to $\p_{r+1}$ is called the {\em completed $(r+1)$-cycle} and denoted
by~$\oC_{r+1}$.
\end{definition}

The first completed cycles are:
$$
\begin{array}{ccl}
\oC_1&=& C_1,\\
\oC_2&=& C_2,\\
\oC_3&=& C_3 + C_{1,1} + \frac1{12} C_1,\\
\oC_4&=& C_4 + 2 C_{2,1} + \frac54 C_2,\\
\oC_5&=& C_5 + 3 C_{3,1} + 4 C_{2,2} + \frac{11}3 C_3
         + 4 C_{1,1,1} + \frac32 C_{1,1} + \frac1{80} C_1.
\end{array}
$$

For reasons that will become clear later, we say that a stable
center element $C_\lambda$ involved in the completed cycle
$\oC_{r+1}$ has {\em genus defect} $[r+2 - \sum (\lambda_i + 1)]/2$.

\subsection{Completed Hurwitz numbers}

Let $K = \sum k_i$. Recall that the completed $(r+1)$-cycle
can be considered as a central element of the group
algebra $\C S_K$. An {\em $r$-factorization
of type $(k_1, \dots, k_n)$} in the symmetric group $S_K$ is a  factorization
$$
\sigma_1 \dots \sigma_m = \sigma
$$
such that 
(i)~the cycle lengths of~$\sigma$ equal $k_1, \dots, k_n$ and
(ii)~each permutation $\sigma_i$ enters the completed $(r+1)$-cycle
with a nonzero coefficient.
The product of these coefficients for $i$ going from 1 to~$m$
is called the {\em weight} of the $r$-factorization.

Choose $m$ points $y_1, \dots, y_m \in \C$ and a
system of $m$ loops $s_i \in \pi_1(\C \setminus \{ y_1, \dots, y_m \})$,
$s_i$ going around $y_i$. Then
to an $r$-factorization one can assign a family of stable maps 
from nodal curves to $\CP^1$. This is done in the following way.
(i)~Consider the covering of $\CP^1$ ramified over $y_1, \dots, y_m$,
and $\infty$ with monodromies given by
$\sigma_1, \dots, \sigma_m$ and $\sigma^{-1}$ (relative to the
chosen loops).
(ii)~If $\sigma_i$ has $l_i$ distinguished cycles and genus 
defect~$g_i$, glue a curve of genus $g_i$ with
$l_i$ marked points to the $l_i$ preimages of the $i$th
ramification point that correspond to the distinguished cycles.
The covering mapping is extended on this new component by 
saying that it is entirely projected to the $i$th ramification
point.
(iii)~Among the newly added components, contract those that are
unstable.

One can easily check that the arithmetic genus of the curve~$C$
constructed in this way is equal to~$g$.
The complex structure on the newly added components of~$C$ 
can be chosen arbitrarily, which implies that in general we obtain
not a unique stable map, but a family of stable maps. 

An $r$-factorization is called {\em transitive} if the curve~$C$
assigned to the factorization is connected, in other words if
one can go from every element of $\{1, \dots, K \}$ to
any other by applying the permutations $\sigma_i$ and jumping
from one distinguished cycle of $\sigma_i$ to another one.

\begin{definition}
The {\em completed Hurwitz number} $h_{g,r;k_1, \dots, k_n}$
is the sum of weights of the transitive $r$-factorizations of type
$(k_1, \dots, k_n)$.
\end{definition}

\begin{theorem}[Okounkov, Pandharipande~\cite{OkoPan1}] \label{Thm:GWHurwitz}
The relative Gromov-Witten invariant
$$
(r!)^m
\!\! 
\int\limits_{\oM^\rel} 
\!\!\!\!
\ev_1^*(\omega) \psi_1^r \cdots \ev_n^*(\omega)\psi_m^r
$$
is equal to the corresponding completed Hurwitz number.
\end{theorem} 
Thus there is no clash of notation if we denote both by
$h_{g,r;k_1, \dots, k_n}$.

\subsection{A digression on Kadomtsev-Petviashvili}
The Hurwitz numbers can be arranged into a generating series
$$
G_r(\beta; p_1, p_2, \dots)
= 
\sum_{g,n} \frac1{n!} \sum_{k_1, \dots, k_n}
h_{g,r;k_1,\dots, k_n} 
\frac{\beta^m}{m!} p_{k_1} \dots p_{k_n}.
$$
(One can prove that $h_{g,r;k_1,\dots, k_n} =0$
whenever $m = (\sum k_i + n + 2g - 2)/r$ is not an integer.)

\begin{theorem} \label{Thm:KP}
The series $G_r$ is a solution of the Ka\-dom\-tsev-Pet\-viash\-vi\-li
(KP) hierarchy in variables $p_i$, $\beta$ being a parameter.
\end{theorem}

The proof of this theorem is a straightforward generalization of~\cite{KazLan} and~\cite{Okounkov}.

\section{The class $\cS$} \label{Sec:Chiodo}

\subsection{Chiodo's formula}
Chiodo computed the Chern characters 
$$
\ch_k(R^\bullet p_*\cL) = \ch_k(R^0 p_*\cL) - \ch_k(R^1 p_*\cL)
$$
and obtained the following expression.

\begin{theorem}[Chiodo~\cite{Chiodo2}, Theorem~1.1.1] 
\label{Thm:Chiodo} 
We have
\begin{align}
\ch_k(R^\bullet p_*\cL) =
&
\frac{B_{k+1}(\frac1r)}{(k+1)!} \kappa_k
- \sum_{i=1}^n 
\frac{B_{k+1}(\frac{a_i+1}r)}{(k+1)!} \psi_i^k
\\ \notag &
+ \frac{r}2 \sum_{a=0}^{r-1} 
\frac{B_{k+1}(\frac{a_i+1}r)}{(k+1)!} (j_a)_* 
\frac{(\psi')^k + (\psi'')^k}{\psi'+\psi''}.
\end{align}
Here $j_a$ is the push-forward from the boundary components with a chosen node and a chosen branch of index~$a$.
\end{theorem}

Chiodo's formula makes it possible to compute any class of the form
$$
\exp\left( \sum_{k \geq 1} s_k \ch_k(R^\bullet p_*\cL) \right).
$$
Specifically, we will need
$$
c(-R^\bullet p_*\cL) = \exp\left( -\sum_{k \geq 1} (k-1)! \ch_k(R^\bullet p_*\cL) \right),
$$
with $s_k = -(k-1)!$.

\subsection{Topological field theory and $R$-matrix} \label{Ssec:CohFT}
In this section we use the notion of a {\em cohomological field theory}, {\em topological field theory} and $R$-matrix action. See~\cite{PiPaZv} or~\cite{BCOV} for an introduction.

Let $\Omega_{g,n}(a_1, \dots, a_n) = \pi_*(\cS)$ be the push-forward of the class $\cS = c(R^1p_*\cL)/c(R^0p_*\cL)$ from the space of $r$-spin structures to $\oM_{g,n}$. Further, denote by $\omega_{g,n}$ the degree~0 part of $\Omega_{g,n}$.

\subsubsection{The topological field theory} \label{sec:flatbasis}
The degree~0 class $\omega_{g,n}$ is equal to the push-forward of the cohomology class~1 from the space of $r$-spin structures to the space of stable curves. Thus it is east to compute:
$$
\omega_{g,n}(a_1, \dots, a_n) = r^{2g-1} \delta_{2g-2-\sum a_i \bmod r},
$$
because the degree of $\pi: \oM^{1/r} \to \oM_{g,n}$ equals~$r^{2g-1}$. The classes $\omega_{g,n}$ form a topological field theory with unit on the vector space 
$\la  e_0, \dots, e_{r-1} \ra$ with quadratic form 
$$
\eta_{ab} = \frac1r \delta_{a+b+2 \bmod r}
$$
and unit $e_0$. 

The 3-point correlators in genus~0 are equal to 
$$
\la e_a e_b e_c \ra = \frac1r \delta_{a+b+c+2 \mod r}.
$$
Therefore the quantum product is given by
$$
e_a \bullet e_b = e_{a+b \bmod r}.
$$
The idempotents of the quantum product are
$$
\frac1r
\sum_{a=0}^{r-1} \J^{ai} e_a,
$$
where $\J$ is the primitive $r$th root of unity and $0 \leq i \leq r-1$. 

\subsubsection{The $R$-matrix and the correlators} \label{sec:Omega}
It follows from Chiodo's formula that the family of cohomology classes $\Omega_{g,n}$ can be obtained from $\omega_{g,n}$ by Givental's $R$-matrix action, where
$$
R(z) = \exp 
\left[-
\sum_{k \geq 1} \frac{
\diag_{a=0}^{r-1} \, B_{k+1}\left(\frac{a+1}{r}\right)
}{k(k+1)} z^k 
\right]
$$
in basis $(e_a)$.

This fact has been observed in~\cite{ChiZvo} in the presence of Witten's class, then in~\cite{ChiRua} in the setting we use here. 
It follows that the classes $\Omega_{g,n}$ form a semi-simple cohomological field theory.

We define the {\em correlators} $\la \tau_{d_1}^{a_1} \cdots \tau_{d_n}^{a_n} \ra^{\! \widetilde{\mathrm{coh}}}_{\!g}$ of this cohomological field theory:
\begin{equation}
\la \tau_{d_1}^{a_1} \cdots \tau_{d_n}^{a_n} \ra^{\!\widetilde{\mathrm{coh}}}_{\! g}:=
\!\!\!
\int\limits_{\oM_{g,n}}  \!\!\!\!
\Omega_{g,n}(a_1, \dots, a_n) \,
\psi_1^{d_1} \cdots \psi_n^{d_n}
=
\!\!\!\!\!\!\!\!\!\!
\int\limits_{\oM^{1/r}_{g,n;a_1, \dots, a_n}}
\!\!\!\!\!\!\!\!\!\!\!\!
\cS \, \psi_1^{d_1} \cdots \psi_n^{d_n} , 
\end{equation}
which are just the coefficients of the $n$-point function
\begin{align} \label{eq:FgrCor}
& F_{g,r}(x_1, \dots, x_n) 
\\ \notag
& 
= 
\!\!\!\!\!
\sum_{
\substack{d_1, \dots, d_n \geq 0 \\ 0 \leq a_1, \dots, a_n \leq r-1}
} 
\!\!\!\!\!
\la \tau_{d_1}^{a_1} \cdots \tau_{d_n}^{a_n} \ra^{\!\widetilde{\mathrm{coh}}}_{\! g}
r^{2g+2n-2+\frac{2g-2-\sum_{i=1}^n a_i}{r}-\sum_{i=1}^n d_i}
\\ \notag
& \phantom{ =r^{2g+2n-2}
\sum
}
\prod_{i=1}^n
\sum_{p_i=0}^{\infty}
\frac{\left(rp_i + r- a_i-1\right)^{p_i+d_i}}{p_i!}
e^{(rp_i+r-a_i-1)x_i} .
\end{align}

\section{Equivalence of the $r$-ELSV and $r$-BM conjectures} \label{sec:SCHur}

In this section we derive a formula for the $n$-point differentials $\cW_{g,n}$ of the curve $\Cs^{(r)}\colon x=-y^r+\log y$ and prove Theorem~\ref{Thm:Main}. In particular, this implies that the $r$-ELSV conjecture and the $r$-BM conjecture are equivalent.

To do that, we perform a number of local computations with the spectral curve, following the recipe of~\cite{DOSS12} in order to present the cohomological field theory corresponding to the spectral curve as a particular $R$-matrix action in the sense of Givental, applied to a topological field theory equal to a direct sum of $r$ one-dimentional topological field theories, properly rescaled. 
We prove that the cohomological field theory that we obtain this way is a multiple of $\{\Omega_{g,n}\}$ defined in Section~\ref{sec:Omega} rewritten in the basis of normalized idempotents.

\begin{proof}[Proof of Theorem~\ref{Thm:Main}.] 
Our goal is to compare the coefficients of the $n$-differentials $\cW_{g,n}$ with the correlators $\left< 
\tau_{d_1}^{a_1} \cdots \tau_{d_n}^{a_n}
\right>_{\! g}^{\! \widetilde{\mathrm{coh}}}$ of the cohomological field theory described in Section~\ref{Ssec:CohFT} in terms of the Givental group action. In~\cite{DOSS12} the coefficients of the differentials $\cW_{g,n}$ are, under some conditions and modulo some extra factors, expressed in terms of correlators of a cohomological field theory. We use the result of~\cite{DOSS12} and compare the cohomological field theories and the extra factors that appear in the statement of the theorem and in the identification formula in~\cite{DOSS12}. 

Let us outline this comparison. We have the following formula for $\cW_{g,n}$ given in~\cite{DOSS12}:
\begin{equation}\label{eq:n-pt-function-0}
\cW_{g,n}=\sum_{
\begin{smallmatrix}
i_1,\dots,i_n \\
d_1,\dots,d_n
\end{smallmatrix}}
\la 
\tau_{d_1}^{i_1} \cdots \tau_{d_n}^{i_n}
\ra_{\!g}^{\!\mathrm{t.r.}} D \left(\prod_{j=1}^n\left(-2\frac{\d}{\d x_j}\right)^{d_i}\xi_{i_j}(x_j)\right),
\end{equation}
where~$\xi_i$ are some auxilliary functions of the coordinate $x$ on the curve, we recall them below. The correlators $\la 
\tau_{d_1}^{i_1} \cdots \tau_{d_n}^{i_n}
\ra_{\!g}^\mathrm{\!t.r.}$ are represented as a sum over Givental graphs, whose structure is described in~\cite{PSL,DOSS12}.

\begin{remark} The local computations that we do in this section in order to specify all ingredients of Equation~\eqref{eq:n-pt-function-0} are a direct generalization of the computations of Eynard in~\cite[Section 8.2]{Ey11} for the curve $x=-y+\log y$ that shows that the usual ELSV formula is equivalent to the Bouchard-Mari\~no conjecture for the usual Hurwitz numbers.
\end{remark}

Note that the indices $i_1,\dots,i_n$ that we have in Equation~\eqref{eq:n-pt-function-0} refer to a summation in the basis of the normalized idempotents (cf. Section~\ref{Ssec:CohFT}), while the indices $a_1, \dots, a_n$ in Theorem~\ref{Thm:Main} refer to a summation in the natural flat basis of the cohomological field theory. \begin{notation}
Throughout this text, when we write an object with a tilde we mean refer to this object expressed in the natural flat basis, while when we write the same object without a tilde it should be interpreted as expressed in the basis of canonical idempotents. Furthermore, $\langle \cdot \rangle^{\mathrm{t.r.}}$ refers to the correlators of the cohomological field theory associated to the the spectral curve~$\Cs^r$ by the methods of~\cite{DOSS12}, whereas~$\langle \cdot \rangle^\mathrm{coh}$ refers to the correlators of the cohomological field theory~$\Omega_{g,n}$. 
\end{notation}
Changing the coordinates, we can rewrite Equation~\eqref{eq:n-pt-function-0} in the following form:\begin{equation}\label{eq:n-pt-function-1}
\cW_{g,n}=\sum_{
\begin{smallmatrix}
a_1,\dots,a_n \\
d_1,\dots,d_n
\end{smallmatrix}}
\la 
\tau_{d_1}^{a_1} \cdots \tau_{d_n}^{a_n}
\ra_{\!g}^{\!\widetilde{\mathrm{t.r.}}} D \left(\prod_{j=1}^n\left(-2\frac{\d}{\d x_j}\right)^{d_i}\tilde\xi_{a_j}(x_j)\right),
\end{equation}
where
\begin{equation}
\la 
\tau_{d_1}^{a_1} \cdots \tau_{d_n}^{a_n}
\ra_{\! g}^{\! \widetilde{\mathrm{t.r.}}} :=
\sum_{
\begin{smallmatrix}
i_1,\dots,i_n
\end{smallmatrix}
}
\la 
\tau_{d_1}^{i_1} \cdots \tau_{d_n}^{i_n}
\ra_{\!g}^{\!\mathrm{t.r.}}
\prod_{j=1}^n \J^{(a_j+1)i_j} \ ,
\end{equation}
and
\begin{equation} \label{eq:xi-a}
\tilde \xi_a := \sum_{i=0}^{r-1} r^{-1} \J^{-(a+1)i} \xi_i \  .
\end{equation}

Below, in Lemma~\ref{lem:leaves}, we prove that
\begin{equation}
\tilde \xi_a = \I \sqrt 2 r^{\frac 12 -\frac{a+1}r} \sum_{n=0}^\infty \frac{ (rn+r-a-1)^{n}}{n!} 
e^{(rn+r-a-1)x}.
\end{equation}

Further, in Lemma~\ref{lem:mainlemma} we prove that
\begin{align}
& \la 
\tau_{d_1}^{a_1} \cdots \tau_{d_n}^{a_n}
\ra_{\!g}^{\! \widetilde{\mathrm{t.r.}}}
\prod_{j=1}^n \I \sqrt {2r} \cdot r^{\frac{-(a_j+1)}r} (-2)^{d_n}
\\ \notag
& =\la
\tau_{d_1}^{a_1} \cdots \tau_{d_\ell}^{a_\ell}
\ra_{\!g}^{\!\widetilde{\mathrm{coh}}}
r^{2g+2\ell-2+\frac{2g-2-\sum_{i=1}^\ell a_i}{r}-\sum_{i=1}^\ell d_i}
\end{align}
Substituting these two expression in Equation~\eqref{eq:n-pt-function-1}, using equation~\eqref{eq:FgrCor}, we obtain the equality of Theorem~\ref{Thm:Main}, which proves the theorem.
\end{proof}

The rest of the section consists of a step-by-step computation of the expansions of various local objects on the curve $x=-y^r+\log y$ that are needed to formulate and prove Lemmas~\ref{lem:leaves} and~\ref{lem:mainlemma}. We follow the scheme of computations proposed in~\cite{DOSS12} and use the same notation as there. 

\begin{equation}
y_i := r^{-1/r} \J^i, \qquad i=0,\dots, r-1.
\end{equation}
The critical values of the function $x$ at these points are 
\begin{equation}
x_i:= -\frac{1}{r} + \frac{\I 2\pi i}{r} - \frac{\log r}{r}, \qquad i=0,\dots, r-1.
\end{equation}
We denote by $z_i$ the local coordinates near the critical points, that is, $z_i^2+x_i=x$, $i=0,1,\dots,r-1$. Let us make a choice of the expansion of function $y$ in $z_i$. One of the possible choices, that fixes $y=y(z_i)$ unambiguously is
\begin{equation}
y(z_i)=y_i+y_{i1} z_i +O(z_i^3),
\end{equation}
where
\begin{equation}
y_{i1} := \I \sqrt 2 r^{-\frac{1}{2} -\frac{1}{r}} \J^i, \qquad i=0,\dots, r-1.
\end{equation}

\subsection{Reciprocal Gamma function}

We are using the following expansion of the reciprocal Gamma function:
\begin{align}
\frac{1}{\Gamma(w+v)} & = \frac{\I}{2\pi} \int_{C_+} (-t)^{-w-v} e^{-t} \dd t 
\\ \notag
& \sim \frac{w^{-w-v+\frac 12} e^w}{\sqrt{2\pi}} \exp\left(\sum_{j=1}^\infty \frac{B_{j+1}(v)}{j(j+1)}
(-w)^{j}\right)\ ,
\end{align}
where $C_+$ is the Hankel contour~\cite{Luke} that goes around the positive real numbers, and 
$B_{j}(v)$ are the Bernoulli polynomials defined by
\begin{equation}
\frac{we^{wv}}{e^w-1} = \sum_{j=0}^\infty B_j(v)\frac{w^j}{j!}\ .
\end{equation}
Note that $B_j(1-v)=(-1)^j B_j(v)$, $j=0,1,\dots$.

\subsection{Local expansions of $y$}

Let us fix $0\leq i\leq r-1$. We are interested in the expansion of the odd part of $y=y(z_i)$ in $z_i$, that is, we consider $(y(z_i)-y(-z_i))/2$. We consider also the coordinate $t:=ry^r$. It is easy to see that 
\begin{align}
y & = t^{\frac 1r} r^{-\frac 1r} \J^i; \\
-x_i-y^r+\log y & = \frac{1}{r}-\frac{t}{r} +\frac{\log t}{r}; \\
\dd y & = t^{\frac{1-r}{r}} r^{-1-\frac{1}{r}} \J^i \dd t.
\end{align}
Note also that $t$ runs along the negative Hankel contour when $z_i$ runs from $-\infty$ to $+\infty$.

\begin{lemma} \label{lem:y} Denote the coefficients of the odd part of~$y(z_i)$ by:
\begin{equation}
\frac{y(z_i)-y(-z_i)}{2} = \I \sqrt 2 r^{-\frac{1}{2} -\frac{1}{r}} \J^i z_i \sum_{j=0}^\infty \frac{V_j (2r)^j z_i^{2j}}{(2j+1)!!}\ . 
\end{equation}
Then we have
\begin{equation}
\sum_{j=0}^\infty V_j \zeta^{j} = \exp\left(-\sum_{i=1}^\infty \frac{B_{j+1}(\frac{1}{r})}{j(j+1)}\zeta^j\right)\ . 
\end{equation}
\end{lemma}

\begin{proof} The Laplace method gives the following asymptotic expansion: 
\begin{align}
\int_{-\infty}^\infty y'(z_i) \exp(-rwz_i^2) \dd z_i & =  \I \sqrt {2\pi} r^{-1-\frac{1}{r}} \J^i w^{-\frac 12}\sum_{j=0}^\infty V_j w^{-j}\ .
\end{align}

On the other hand, 
\begin{align}
\int_{-\infty}^\infty y'(z_i) e^{-rwz_i^2} \dd z_i
& = \int_{\tilde C_-} e^{-rw\left(-x_i-y^r+\log(y)\right)} \dd y 
\\ \notag
& = r^{-\frac{1+r}{r}}\J^i\int_{C_-} e^{-w(1-t+\log t)} t^{\frac{1-r}{r}} \dd t
\\ \notag
& =  r^{-\frac{1+r}{r}}\J^i e^{-w} \int_{C_-} e^{wt} t^{-w+\frac{1-r}{r}}\dd t
\\ \notag
& =  -r^{-\frac{1+r}{r}}\J^i e^{-w} w^{w-\frac 1r} \int_{C_+} e^{-p} (-p)^{-w+\frac{1-r}{r}}\dd p
\end{align}
(we used the substitution $p=-wt$ that transforms the negative Hankel contour $C_-$ to the Hankel contour $C_+$). Thus we see that
\begin{align}
& \int_{-\infty}^\infty y'(z_i) e^{-rwz_i^2} \dd z_i
=  \frac{2\pi \I r^{-\frac{1+r}{r}}\J^i e^{-w} w^{w-\frac 1r}}{\Gamma (w+1-\frac{1}{r})}
\\ \notag
& \sim 
\frac{2\pi \I r^{-\frac{1+r}{r}}\J^i e^{-w} w^{w-\frac 1r}w^{-w+\frac 1r-\frac 12} e^w}{\sqrt{2\pi}} \exp\left(\sum_{j=1}^\infty \frac{B_{j+1}(1-\frac 1r)}{j(j+1)}
(-w)^{-j}\right)
\\ \notag
&= \sqrt{2\pi} \I r^{-\frac{1+r}{r}}\J^i w^{-\frac 12} \exp\left(-\sum_{j=1}^\infty \frac{B_{j+1}(\frac 1r)}{j(j+1)}
w^{-j}\right)\ ,
\end{align}
which proves 
\begin{equation}
\sum_{j=0}^\infty V_j w^{-j} = \exp\left(-\sum_{j=1}^\infty \frac{B_{j+1}(\frac{1}{r})}{j(j+1)}w^{-j}\right)\ . 
\end{equation}
\end{proof}

\subsection{Two-point function}

Now we consider the two-point function. According to~\cite{Ey11}, since $\dd x$ is a meromorphic $1$-form in $y$, we know that the Laplace transform of the two-point function is represented as a Givental-type edge contribution. So, we have to compute only the even coefficients of the local expansion of half of the two-point function in order to specify the Givental operator imposed by the topological recursion, namely, we are interested in the coefficients of the function
\begin{equation}
Y_{i_1i_2}:=\frac{y_{i_11} y'(z_{i_b})}{(y_{i_1}-y(z_{i_2}))^2}=\frac{\delta_{i_1i_2}}{z^2}+O(1).
\end{equation}

\begin{lemma} \label{lem:2p} We have:
\begin{equation}\label{eq:lem2-1}
\frac{Y_{{i_1}{i_2}}(z_{i_2})+Y_{{i_1}{i_2}}(-z_{i_2})}{2} = -\sum_{k=0}^\infty \frac{(U_k)_{{i_1}{i_2}} (2r)^i z_{i_2}^{2k-2}}{(2k-3)!!}\ , 
\end{equation}
where~$(U_k)_{{i_1}{i_2}}$ is given by 
\begin{equation}\label{eq:lem2-2}
\sum_{k=0}^\infty (U_k)_{{i_1}{i_2}} z^{k} = \frac{1}{r}\sum_{c=0}^{r-1}
\exp\left(-\sum_{k=1}^\infty \frac{B_{k+1}(\frac{c}{r})z^{k}}{k(k+1)}\right)
 \J^{c{i_2}-c{i_1}}\ . 
\end{equation}
\end{lemma}

\begin{proof} Observe that
\begin{align}
\int_{-\infty}^\infty \left(\frac{y_{{i_1}1} y'(z_{i_2})}{(y_{i_1}-y(z_{i_2}))^2}\right) e^{-rwz_{i_2}^2} \dd z_{i_2} 
& \sim -2\sqrt{\pi} r^{\frac 12}w^{\frac 12} \sum_{k=0}^\infty (U_k)_{{i_1}{i_2}} w^{-k}.
\end{align}

On the other hand,
\begin{align}
& \int_{-\infty}^\infty \frac{y_{{i_1}1} y'(z_{i_2})}{(y_{i_1}-y(z_{i_2}))^2} e^{-rwz_{i_2}^2} \dd z_{i_2}
\\ \notag
& = rw y_{{i_1}1} \int_{-\infty}^\infty \frac{1}{y_{i_1}-y(z_{i_2})} e^{-rwz_{i_2}^2} 2z_{i_2} \dd z_{i_2}
\\ \notag
& = rw y_{{i_1}1} \int_{C_-} \frac{y(z_{i_2})^r - \frac 1r}{y(z_{i_2})-y_{i_1}} e^{-w(1-t+\log t)} \frac{\dd t}{t} \ .
\end{align}
Here we can use that
\begin{align}
 ry_{{i_1}1}\frac{y(z_{i_2})^r - \frac 1r}{y(z_{i_2})-y_{i_1}} &
= r\I\sqrt 2 r^{-\frac 12 -\frac 1r} \J^{i_1}\frac{(t^{\frac 1r} r^{-\frac 1r} \J^{i_2})^r-r^{-1}}
{t^{\frac 1r} r^{-\frac 1r} \J^{i_2} - r^{-\frac 1r} \J^{i_1}} \\ \notag
& = \I\sqrt 2 r^{-\frac 12} \sum_{c=0}^{r-1} \J^{c{i_2}-c{i_1}} t^{\frac cr}\ .
\end{align}
Therefore,
\begin{align}
& \int_{-\infty}^\infty \frac{y_{{i_1}1} y'(z_{i_2})}{(y_{i_1}-y(z_{i_2}))^2} e^{-rwz_{i_2}^2} \dd z_{i_2}
\\ \notag
& = \I\sqrt 2 r^{-\frac 12} we^{-w} \sum_{c=0}^{r-1} 
\J^{c{i_2}-c{i_1}}
\int_{C_-} e^{wt}t^{-w-1+\frac cr} \dd t 
\\ \notag
& =  -\I\sqrt 2 r^{-\frac 12} w^{w+1}e^{-w} \sum_{c=0}^{r-1} 
\J^{c{i_2}-c{i_1}} w^{-\frac cr}
\int_{C_+} e^{-p}(-p)^{-w-1+\frac cr} \dd p
\\ \notag
& \sim -2\sqrt \pi r^{-\frac 12} w^{\frac 12} \sum_{c=0}^{r-1} 
\J^{c{i_2}-c{i_1}} \exp \left(-\sum_{n=1}^\infty \frac{B_{n+1}(\frac cr)}{n(n+1)} w^{-c} \right) \ .
\end{align}

Thus we see that
\begin{equation}
\sum_{k=0}^\infty (U_k)_{{i_1}{i_2}} z^{k} = \frac{1}{r}\sum_{c=0}^{r-1}
\exp\left(-\sum_{k=1}^\infty \frac{B_{k+1}(\frac{c}{r})z^{k}}{k(k+1)}\right)
 \J^{c{i_2}-c{i_1}}\ . 
\end{equation}
\end{proof}

\subsection{Functions on the leaves}

According to~\cite{DOSS12}, the auxilliary function $\xi_i(x)$ that we put on the leaves in the graph expression for the correlation forms of the spectral curve are given by the following formula:
\begin{equation}
\xi_i:= \frac{\I \sqrt 2 r^{-\frac{1}{2} -\frac{1}{r}} \J^i}{r^{-\frac 1r} \J^i-y}.
\end{equation}
Here the index $i$ corresponds to the basis of normalized idempotents, so in the standard flat  basis we have to consider the functions $\tilde\xi_a$ given by Equation~\eqref{eq:xi-a}:
\begin{equation}
\tilde \xi_a := \sum_{i=0}^{r-1} r^{-1} \J^{-(a+1)i} \xi_i \  .
\end{equation}

\begin{lemma}\label{lem:leaves} We have:
\begin{equation}
\tilde \xi_a = \I \sqrt 2 r^{\frac 12 -\frac{a+1}r} \sum_{n=0}^\infty \frac{ (rn+r-a-1)^{n}}{n!} e^{(rn+r-a-1)x}\ .
\end{equation}
\end{lemma}

\begin{proof} First, observe that 
\begin{align}
\tilde \xi_a & = \sum_{i=0}^{r-1} r^{-1} \J^{-(a+1)i} \xi_i 
=
\I \sqrt 2 r^{-1-\frac 12} \sum_{i=0}^{r-1}  \frac{\J^{-(a+1)i}}{1-r^{\frac 1r} \J^{-i} y}
\\ \notag 
& = \I \sqrt 2 r^{-\frac 12} \left(\frac{r^{\frac {r-a-1}r} y^{r-a-1} }{1-r y^r} \right).
\end{align}

Following~\cite{CGHJK96}, we define the Lambert function
\begin{equation}
W(z):=-\sum_{n=1}^\infty \frac{n^{n-1}}{n!} (-z)^n \ ,
\end{equation}
and use its property
\begin{equation}
\left(\frac{W(z)}{z}\right)^\alpha=\sum_{n=0}^\infty \frac{\alpha (n+\alpha)^{n-1}}{n!} (-z)^n \ .
\end{equation}
We have the following equation: $e^x=ye^{-y^r}$. This equation implies (cf.~\cite{MSS})
\begin{equation}
y=\left(\frac{W(-re^{rx})}{-r}\right)^{\frac 1r}
\end{equation}
and
\begin{equation}
\frac{\dd y^{r-a-1}}{\dd x} = \frac{(r-a-1)y^{r-a-1}}{1-ry^r} \ .
\end{equation}
Therefore,
\begin{align}
& \frac{(r-a-1)y^{r-a-1}}{1-ry^r} = \frac{\dd}{\dd x} \left(\frac{W(-re^{rx})}{-r}\right)^{\frac {r-a-1}r}
\\ \notag
& = (-r)^{-\frac {r-a-1}r} \frac{\dd}{\dd x} (-re^{rx})^{\frac {r-a-1}r} 
\sum_{n=0}^\infty \frac{\frac {r-a-1}r (n+\frac {r-a-1}r)^{n-1}}{n!} (re^{rx})^n 
\\ \notag
&= (r-a-1)  \frac{\dd}{\dd x}
\sum_{n=0}^\infty \frac{ (rn+r-a-1)^{n-1}}{n!} e^{(rn+r-a-1)x}
\\ \notag
& = (r-a-1) \sum_{n=0}^\infty \frac{ (rn+r-a-1)^{n}}{n!} e^{(rn+r-a-1)x}
\end{align}
So, we see that
\begin{equation}
\tilde \xi_a = \I \sqrt 2 r^{\frac 12-\frac{a+1}r} \sum_{n=0}^\infty \frac{ (rn+r-a-1)^{n}}{n!} e^{(rn+r-a-1)x}
\end{equation}
\end{proof}

\subsection{Comparison of the correlators}

Consider the $n$-point correlators $\la 
\tau_{d_1}^{a_1} \cdots \tau_{d_n}^{a_n}
\ra_{\!g}^{\!\widetilde{\mathrm{t.r.}}} $ introduced in Equation~\eqref{eq:n-pt-function-1}.
They are obtained via a linear change of the indices from the correlators $\la 
\tau_{d_1}^{i_1} \cdots \tau_{d_n}^{i_n}
\ra_{\!g}^{\!\mathrm{t.r.}}$, where the latter ones are defined in~\cite{DOSS12} as the sum over Givental graphs. The structure constants of these graphs, that is, the parameters that we put on vertices, edges, leaves, and dilaton leaves, are defined in terms of local data of the curve $x=-y^r+\log y$ at the ramification points. More precisely, they are defined via the coefficients of the expansions of the function $y(z_i)$ in the local coordinate $z_i$ and the components of the Bergman kernel $Y_{i_1i_2}$ in the local coordinate $z_{i_2}$. 

In this Section we prove that the correlators  $\la\tau_{d_1}^{a_1} \cdots \tilde\tau_{d_n}^{a_n}
\ra_{\!g}^{\widetilde{\!\mathrm{t.r.}}}$
are equal, up to some multiplicative factors, to the correlators $\la 
\tau_{d_1}^{a_1} \cdots \tau_{d_n}^{a_n}
\ra_{\!g}^{\widetilde{\!\mathrm{coh}}}$ of the cohomological field theory described in Section~\eqref{Ssec:CohFT}.  Namely, we prove the following lemma.

\begin{lemma} \label{lem:mainlemma}
\begin{align}
& \sum_{
\begin{smallmatrix}
i_1,\dots,i_\ell
\end{smallmatrix}}
\la 
\tau_{d_1}^{i_1} \cdots \tau_{d_n}^{i_n}
\ra_{\!g}^{\!\mathrm{t.r.}}
\prod_{k=1}^n r^{\frac 12} \J^{(a_k+1)i_k} 
\prod_{k=1}^n \I \sqrt 2 r^{\frac{-(a_k+1)}r} (-2)^{d_k}
\\ \notag
& =\la
\tau_{d_1}^{a_1} \cdots \tau_{d_n}^{a_n}
\ra_{\!g}^{\!\widetilde{\mathrm{coh}}}
r^{2g+2n-2+\frac{2g-2-\sum_{k=1}^n a_k}{r}-\sum_{k=1}^n d_i}
\end{align}
\end{lemma}

\begin{proof} The proof follows from the comparison of the ingredients of the Givental graph expressions on both sides, following the identification theorem in~\cite{DOSS12}. Let us rewrite both sides of the equality in the basis of normalized idempotents:
\begin{align} \label{eq:ncf}
& 
\la 
\tau_{d_1}^{i_1} \cdots \tau_{d_n}^{i_n}
\ra_{\!g}^{\!\mathrm{t.r.}}
\prod_{k=1}^n (-2r)^{d_k+\frac 12}
=\la
\tau_{d_1}^{i_1} \cdots \tau_{d_n}^{i_n}
\ra_{\!g}^{\!\mathrm{coh}}
r^{2g+n-2+\frac{2g+n-2}{r}},
\end{align}
where
\begin{equation}
\la
\tau_{d_1}^{i_1} \cdots \tau_{d_n}^{i_n}
\ra_{\!g}^{\!\mathrm{coh}}:=
\sum_{
\begin{smallmatrix}
a_1,\dots,a_n
\end{smallmatrix}}
\la 
\tau_{d_1}^{a_1} \cdots \tau_{d_n}^{a_n}
\ra_{\!g}^{\!\widetilde{\mathrm{coh}}}
\prod_{j=1}^n \J^{-(a_j+1)i_j} 
\end{equation}

The result of Chiodo (see Theorem~\ref{Thm:Chiodo}) implies that the generating function of the correlators $\la
\tau_{d_1}^{i_1} \cdots \tau_{d_n}^{i_n}
\ra_{\!g}^{\!\mathrm{coh}}$ is obtained from the $r$ properly normalized copies of the Kontsevich-Witten tau function by application of the quantization of the operator
\begin{equation}
R_{i}^j(\zeta):=\exp\left(-\sum_{k=1}^\infty
\frac{\zeta^k}r  \sum_{a=0}^{r-1} 
\J^{ai-aj} 
\frac{B_{k+1}\left(\frac{a}r\right)}{k(k+1)} \right) \ .
\end{equation}
In particular (we use it below), we have:
\begin{align}
R_{i}^\mathbf{1}(\zeta)& :=\sum_{j=0}^{r-1} \sum_{a=0}^{r-1} \frac{\J^{j}}{r} \frac{\J^{ai-aj}}{r}  \exp\left(-\sum_{k=1}^\infty
\zeta^k  
\frac{B_{k+1}\left(\frac{a}r\right)}{k(k+1)} \right) \\ \notag
& =\frac{\J^{i}}{r}  \exp\left(-\sum_{k=1}^\infty
\zeta^k  
\frac{B_{k+1}\left(\frac{1}r\right)}{k(k+1)} \right) \ .
\end{align}

The weight of the correlators $\left< \prod_{i=1}^p \tau_{a_i}\right>_q$ of the $i$-th copy of the Kontsevich-Witten tau function (or, in other words, the weight of the vertex labelled by $i$ in the graphical representation of the Givental formula, as in~\cite{DOSS12}) is equal to 
\begin{equation}\label{eq:weight}
r^{2q-1}\!\!\!\!\!\!\!\!\!\!\! \sum_{
\begin{smallmatrix}
a_1,\dots, a_p:\\
r|2g-2-a_1+\cdots+a_p
\end{smallmatrix}
}
\prod_{j=1}^p \J^{-(a_j+1)i}
=
r^{2q+p-2}\J^{-(2q+p-2)i}\ .
\end{equation}

Let us compare that with the formula we get from the topological recursion, following the lines of~\cite{DOSS12}. We compare the coefficients of the expansion of $y$ and the two-point function in the coordinate $z_i$, $i=0,\dots,r-1$ (that determine the ingredients of the graphs in the formula for $\la 
\tau_{d_1}^{i_1} \cdots \tau_{d_n}^{i_n}
\ra_{\!g}^{\!\mathrm{t.r.}}$) with the corresponding formulas in terms of the operator $R_i^j(\zeta)$ that are used in the Givental graphical formula for $\la
\tau_{d_1}^{i_1} \cdots \tau_{d_n}^{i_n}
\ra_{\!g}^{\!\mathrm{coh}}$, in the same way as it is done in~\cite[Theorem 4.1]{DOSS12}. 

Lemma~\ref{lem:y} implies that, in the notation of~\cite{DOSS12},
\begin{align}
\check h_{k+1}^i & = \I 2 \sqrt 2 r^{-\frac{1}{2} -\frac{1}{r}} \J^i (2r)^k [\zeta^k]
\exp\left(-\sum_{i=1}^\infty \frac{B_{i+1}(\frac{1}{r})}{i(i+1)}\zeta^i\right)  \\ \notag
& =  \I \sqrt {2r} r^{-1-\frac{1}{r}} (-2r)^{k+1} [\zeta^k]
\left(-R_{i}^\mathbf{1}(-\zeta)\right)\ .
\end{align}
Lemma~\ref{lem:2p} implies that, also in notation of~\cite{DOSS12},
\begin{align}
\check B^{ji}_{0,k} & =-(2r)^{k+1} [\zeta^{k+1}] R_i^j(\zeta) =
 (-2r)^{k+1} [\zeta^k] \left(\frac{1-R(-\zeta)}{\zeta}\right).
\end{align}
The vertex labelled by $\langle \prod_{i=1}^p \tau_{a_i}\rangle_q$ and an extra index $i$ is also multiplied by (again in the notation of~\cite{DOSS12})
\begin{equation}
(-2h_1^i)^{2-2q-p}=(-\I 2\sqrt 2 r^{-\frac 12 -\frac 1r} \J^i)^{2-2q-p}.
\end{equation}

This all together (including the factors $(-2r)^{d_n+\frac 12}$ that we have on the global leaves in Equation~\eqref{eq:ncf}) gives the following extra factor for the  vertex labelled by $\langle \prod_{i=1}^p \tau_{a_i}\rangle_q$, with $p_d$ attached dilaton leaves and $p_o$ ordinary leaves and/or half-edges ($p=p_d+p_o$), and an extra index $i$:
\begin{align}\label{eq:coeff}
\frac{(-2r)^{a_1+\cdots+a_p} (\I\sqrt{2r})^p r^{-p_d-\frac {p_d}r} }{(-\I 2\sqrt 2 r^{-\frac 12 -\frac 1r} \J^i)^{2q+p-2}}
=
r^{(1+\frac 1r)(2q+p_o-2)}r^{(2q+p-2)} \J^{(2q+p-2)i}
\end{align}
(we used that $a_1+\cdots+a_p=3q-3+p$). Note that the factor $r^{(2q+p-2)} \J^{(2q+p-2)i}$ coinsides with the weight that we have in the Givental's presentation of Chiodo's formula, cf. Equation~\eqref{eq:weight}.
Meanwhile, the sum of the exponents $2q+p_o-2$ over all vertices in a graph of genus $g$ with $n$ global leaves is equal to $2g+n-2$. Therefore, the product of the factors $r^{(1+\frac 1r)(2q+p_o-2)}$ over all vertices of a graph is exactly equal to the extra factor $r^{2g+n-2+\frac{2g+n-2}{r}}$ that we have on the right hand side of Equation~\eqref{eq:ncf}.
\end{proof}



\section{Completed Hurwitz numbers as a matrix model} \label{sec:HurMM}

In this section we state and prove a theorem, representing the partition function for completed Hurwitz numbers as a matrix integral, that is shown in Section~\ref{Sec:Evidence} to lead to strong evidence for the $r$-BM conjecture. Both the theorem presented in this section and the resulting evidence in the next section are based on a direct generalization of one of the proofs of the original Bouchard-Mari\~no conjecture in~\cite{BEMS}. Unfortunately, it seems that this proof is not completely rigorous, meaning that we can only present the generalization to completed Hurwitz numbers as evidence for the $r$-BM conjecture. 

\subsection{The statement} 
Let~$Z$ be the partition function for Hurwitz numbers with completed $(r+1)$-cycles:
\begin{equation}\label{eq:Z}
Z:= \exp\left(\sum_{g=0}^\infty \frac{g_s^{2g-2}}{m!}\sum_{n\geq 1} \frac{1}{n!}\sum_{k_1, \ldots, k_n} h_{g,r;k_1,\ldots,k_n}\prod_{i=1}^n p_{k_i} \right) .
\end{equation}
The character formula for not necessarily connected Hurwitz numbers allows us to write
\begin{align}
Z(\p,g_s;t) = \sum_{K=0}^{\infty} t^K 
& \sum_{|\mu| = K} \sum_{m=0}^{\infty}  \frac{g_s^{rm-K-l(
\mu)}}{m!}p_{\mu}\times 
\\ \notag &
\sum_{|\lambda| = K} \left(\frac{\dim(\lambda)}{n!}\right)^2\frac{|C_{\mu}|\chi_{\lambda}(\mu)}{\dim(\lambda)}\left(\frac{\p_{r+1}(\lambda)}{(r+1)}\right)^m.
\end{align}
Here $rm-K-|\mu|$ is equal to the Euler characteristic of the curve (by the Riemann-Hurwitz formula), $\lambda$ and $\mu$ are partitions of $K$ encoding an irreducible representation and a conjugacy class $C_\mu$ respectively, and $p_\mu = p_{k_1} \cdots p_{k_n}$ for $\mu = (k_1, \dots, k_n)$. For every $\mu$ the coefficient of $p_\mu$ in this expression is a formal Laurent series in $g_s$ with a finite number of negative degree terms. 

Note that we have inserted an extra formal variable~$t$ to encode the degree of the covering. It is redundant, since the degree can also be recovered from the total degree in~$\mathbf{p}$, but turns out to be convenient later on.

Fix a positive integer $N$. We use the following substitution for the variables $p_k$, $k=1,2,\dots$, as symmetric functions:
\begin{equation}
p_k = g_s \sum_{i=1}^N v_i^k.
\end{equation} 
Moreover, introduce the $N$-tuple of variables $\mathbf{v} := (v_1, \ldots, v_N)$ and the diagonal matrix of their formal logarithms $\mathbf{R}:= \diag(\log v_1, \ldots \log v_N)$. We use $\Delta$ to denote the Vandermonde determinant: 
\begin{equation}
\Delta(\mathbf{v}) := \prod_{1 \leq j < i \leq N}(v_i - v_j) , 
\end{equation}
and similarly for~$\Delta(\mathbf{R})$, where we take the Vandermonde determinant of its diagonal entries.

Let $B_k$ be the Bernoulli numbers, and introduce the following functions depending also on the parameter~$N$:
\begin{align} \label{Eq:A}
A_{r+1}(x) & = \sum_{k=0}^{r+1}\left(r!\frac{(-N+\frac{1}{2})^k}{k!}\frac{x^{r+1-k}}{(r+1-k)!} \right.
\\ \notag & {\ } \qquad
\left.
+ (-1)^{r+1} r!\frac{(-1)^k B_k}{k!} \frac{(2^{k-1} - 1)}{2^{k-1}}\frac{N^{r+1-k}}{(r+2-k)!}\right) ;
\end{align}
\begin{align}\label{eq:potential}
V(x)  = & -g_s^{r+1} A_{r+1}(x/g_s) + g_s\log(g_s/t)A_1(x/g_s) 
\\ \notag
& + \I\pi x - g_s \log(\Gamma(-x/g_s)) +\I\pi g_s.
\end{align}
The function~$A_{r+1}$ is just a function of~$x$, whereas~$V$ is a function of~$x$ that also depends on the variables $g_s$ and~$t$ that live on~$\C\backslash (-\infty, 0)$. These functions originate from the combinatorics of Young diagrams, their meaning will be explained later in this section.

Let $\cC_D$ be a fixed contour in the complex plane that goes around the integers~$h$ with $0 \leq h \leq D$. Let $\mathcal{H}_N(\mathcal{C}_D)$ be the space of $N$ by~$N$ 
\emph{normal matrices} $M$ with eigenvalues in~$\mathcal{C}_D$. In other words, $M \in \mathcal{H}_N(\mathcal{C}_D)$ if and only if $M$ can be diagonalized by conjugation with a unitary matrix, and its eigenvalues belong to~$\mathcal{C}_D$:
\begin{equation}
M = U^\dagger X U , \quad U \in U(N), \quad X = \diag(x_1, \ldots, x_N), \quad x_i \in \mathcal{C}_D.
\end{equation}

We use the following measure on~$\mathcal{H}_N(\mathcal{C}_D)$:
\begin{equation}
\dd M = \Delta(X)^2 \dd X \dd U,
\end{equation}
where $\dd U$ is the Haar measure on $U(N)$ and $\dd X$ is the product of Lebesgue curvilinear measures along~$\mathcal{C}_D$. 

Now we formulate the main theorem of this section.

\begin{theorem} \label{thm:Zmatrix} We have:
\begin{equation}\label{eq:Zmatrix}
Z(\mathbf{p}, g_s;t) \sim \frac{g_s^{-N^2}}{N!} \frac{\Delta(\mathbf{R})}{\Delta(\mathbf{v})} \int_{\mathcal{H}_N(\mathcal{C}_D)} \dd M e^{- \frac{1}{g_s} \Tr (V(M) - M\mathbf{R})}.
\end{equation}
Notice that the left-hand side is a formal Laurent series in $\mathbf{p}, g_s, t$, whereas the right-hand side is a meromorphic function of those variables, defined for $t$ and~$g_s$ on the domain~$\C\backslash (-\infty, 0)$, that also depends on two parameters $N$ and~$D$. The symbol~$\sim$ means that for any~$K$, the coefficient of $t^K$ on the left-hand side is given by the coefficient of $t^K$ in the series expansion around $t=0$ of the function on the right-hand side for any choice of the parameters such that $N > K$ and $D > K + \frac{N-1}{2}$. 
\end{theorem}

\begin{remark}\label{Rem:Zcontour}
The form of Theorem~\ref{thm:Zmatrix} differs from that of the analogous statement in~\cite{BEMS}; here the contour is around a finite set of integers, whereas in~\cite{BEMS} it goes around all non-negative integers. According to our understanding, the contour should be finite in both cases, since the integral over the infinite contour does not converge to a meromorphic function, which makes it impossible to have an expansion for it in powers of~$t$. See Remark~\ref{Rem:ZvsZD} for a more precise discussion of the origin of this problem. Note that this is one of the reasons we are not able to convert the evidence in the next section into a formal theorem (see Section~\ref{SSec:EvidenceMM}). 
\end{remark}

The proof of this theorem occupies the rest of this section.

\subsection{Schur polynomials} We recollect some facts about the Schur polynomials $s_\lambda(\mathbf{v})$ that can be defined, for a sufficiently large $N$, by the following formula:
\begin{equation}
s_\lambda(\mathbf{v}) := \frac{\det(v_i^{\lambda_j - j + N})}{\Delta(\mathbf{v})}\ .
\end{equation}

The Schur polynomials are related to representations of the symmetric group (and thus to Hurwitz numbers) by the Frobenius formula
\begin{equation}
s_{\lambda}({\mathbf v}) = \frac{1}{n!} \sum_{|\mu| = n} |C_\mu|\chi_{\lambda}(C_\mu)\tilde{p}_\mu \quad \text{ (where } \tilde{p}_m = \sum_{i=1}^{l(\mu)} v_i^{m} \text{ ) .} 
\end{equation}

There is an expression for~$s_\lambda$ in terms of the Itzykson-Zuber integral (see~\cite{IZ,BEMS})
\begin{equation}
I(X,Y) := \int_{U(N)} \dd U e^{\Tr(XUYU^\dagger)} = \frac{\det(e^{x_i y _j})}{\Delta(X)\Delta(Y)} ,
\end{equation}
where $\dd U$ is the Haar measure on $U(N)$, normalized according to the second equality. Denote by $\mathbf{h}_\lambda$ the diagonal matrix $\diag(h_1 \ldots h_N)$, where $h_i = \lambda_i - i + N$, and by $\Delta(\mathbf{h}_\lambda)$ the Vandermonde determinant of its diagonal entries. Then we have:
\begin{equation}\label{eq:schur-IZ}
s_\lambda(\mathbf{v}) = \Delta(\mathbf{h}_\lambda)\frac{\Delta(\mathbf{R})}{\Delta(\mathbf{v})} I(\mathbf{h}_\lambda, \mathbf{R}) .
\end{equation}

\subsection{Partition function in terms of the Itzykson-Zuber integral}\label{sec:partIZ}
First, we rearrange the partition function for completed Hurwitz numbers in the following way:
\begin{align}
Z(\mathbf{p},g_s;t) & = 
\sum_{K=0}^{\infty} t^n \sum_{m=0}^{\infty} \frac{g_s^{mr - K}}{m!} \sum_{\begin{smallmatrix}
|\lambda| = K\\ 
l(\lambda) \leq N
\end{smallmatrix}} s_\lambda(\mathbf{v})\frac{\dim(\lambda)}{n!} \left(\frac{\mathbf{p}_{r+1}(\lambda)}{r+1}\right)^m 
\\ \notag
& = 
\sum_{l(\lambda)\leq N} \left(\frac{t}{g_s}\right)^{|\lambda|} \frac{\dim(\lambda)}{n!} s_\lambda(\mathbf{v}) e^{g_s^r \frac{\mathbf{p}_{r+1}(\lambda)}{r+1}}.
\end{align}
Here we use the interpretation of $p_k$, $k=1,2,\dots$, as symmetric functions in $v_i$, $1\leq i\leq N$. Furthermore, the formula above should be interpreted order by order in powers of~$t$; for any given power~$K$ of~$t$, the formula is true for~$N$ larger than~$K$. We should keep this interpretation in mind throughout the rest of the computations.

Suppose that we have found functions $A_{r+1}(x)$ such that
\begin{equation}\label{eq:A-def}
\sum_{i=1}^N A_{r+1}(h_i) = \frac{\mathbf{p}_{r+1}(\lambda)}{r+1} ,
\end{equation}
so that in particular
$$
\sum_{i=1}^N A_1(h_i) = |\lambda|.
$$
Then, applying Equation~\eqref{eq:schur-IZ} and the equality
$$
\frac{\dim(\lambda)}{|\lambda|!}= \frac{\Delta(\mathbf{h})}{\prod_{i=1}^N h_i!}
\quad \mbox{for} \quad
N\geq l(\lambda),
$$ 
we get the following:
\begin{align} \label{eq:Z-Nsum}
&Z(\mathbf{p},g_s;t) 
= 
\frac{\Delta(\mathbf{R})}{\Delta(\mathbf{v})} \sum_\lambda \left(\frac{t}{g_s}\right)^{|\lambda|} I(\mathbf{h}_\lambda, \mathbf{R}) \frac{(\Delta(\mathbf{h}_\lambda))^2}{\prod_{i=1}^N h_i!} \prod_{i=1}^N e^{g_s^r A_{r+1}(h_i)}
\\ \notag
& = 
\frac{\Delta(\mathbf{R})}{\Delta(\mathbf{v})} \sum_{h_1 > \cdots > h_N \geq 0} I(\mathbf{h},\mathbf{R})(\Delta(\mathbf{h}))^2\prod_{i=1}^N \frac{e^{g_s^r A_{r+1}(h_i)} (g_s/t)^{-A_1(h_i)}}{\Gamma(h_i + 1)}
\\ \notag 
& = 
\frac{1}{N!}\frac{\Delta(\mathbf{R})}{\Delta(\mathbf{v})} \sum_{h_1, \ldots, h_N \geq 0} I(\mathbf{h},\mathbf{R})(\Delta(\mathbf{h}))^2\prod_{i=1}^N \frac{e^{g_s^r A_{r+1}(h_i)}(g_s/t)^{-A_1(h_i)}}{\Gamma(h_i + 1)} .
\end{align}

\begin{remark}
Note that we should be careful when writing~$(g_s/t)^{A_1(h)}$, since this might introduce non-integer powers of the formal variables $g_s$ and~$t$. In fact, by equation~\eqref{eq:A1}, we obtain half-integer powers. However, it is clear that in equation~\eqref{eq:Z-Nsum} they will cancel in the full product over~$i$.
\end{remark}

\subsection{Computation of $A_{r+1}$}\label{sec:compA} In this section we compute explicitly the polynomials $A_{r+1}$ using Equation~\eqref{eq:A-def} as a definition. The result will coincide with Equation~(\ref{Eq:A}).


We have
\begin{align}
\mathbf{p}_{r+1}(\lambda)  &= \sum_{i=1}^{N}\left((\lambda_i - i + \frac{1}{2})^{r+1} - (-i + \frac{1}{2})^{r+1}\right)
\notag \\  &
= \sum_{i=1}^{N}\left((h_i - N + \frac{1}{2})^{r+1} - (-i + \frac{1}{2})^{r+1}\right)
\\ \notag &
= \sum_{i=1}^{N}\sum_{k=0}^{r+1}\binom{r+1}{k}(-N + \frac{1}{2})^k h_i^{r+1-k} - \sum_{j=1}^N\left(\frac{-2j+1}{2}\right)^{r+1}.
\end{align}
The second term can be represented in the following form:
\begin{align}
& \sum_{j=1}^N\left(\frac{-2j+1}{2}\right)^{r+1} = \frac{1}{(-2)^{r+1}} \left(\sum_{j=1}^{2N-1} j^{r+1} - \sum_{k=1}^{N-1} (2k)^{r+1}\right)
\\ \notag 
& = \frac{(-1)^{r+1}}{r+2} \sum_{k=0}^{r+1}\binom{r+2}{k} (-1)^k B_k \left(\frac{1}{2^{r+1}} (2N)^{r+2-k} - N^{r+2-k}\right) 
\\ \notag 
& = \frac{(-1)^r}{r+2}\sum_{k=0}^{r+1}\binom{r+2}{k} (-1)^k B_k \left(\frac{(2^{k-1}-1)N^{r+2-k}}{2^{k-1}}\right)
\end{align}
(here $B_k:=B_k(0)$, $k=0,1,\dots$, are the Bernoulli numbers).

Thus we have the following formula for $A_{r+1}$:
\begin{align}
A_{r+1}(x) = & \sum_{k=0}^{r+1}\left(\binom{r+1}{k}(-N+\frac{1}{2})^k\frac{x^{r+1-k}}{r+1} \right. 
\\ \notag & 
\left. +\frac{(-1)^{r+1}}{(r+2)(r+1)}\binom{r+2}{k} (-1)^k B_k \frac{(2^{k-1} - 1) N^{r+1-k}}{2^{k-1}}\right)
\\ \notag 
=  \sum_{k=0}^{r+1}&\left(r! \frac{(-N+\frac{1}{2})^k}{k!}\frac{x^{r+1-k}}{(r+1-k)!} \right.
\\ \notag & 
\left.
+ (-1)^{r+1} r! \frac{(-1)^k B_k}{k!} \frac{(2^{k-1} - 1)}{2^{k-1}}\frac{N^{r+1-k}}{(r+2-k)!}\right)
\end{align}
in agreement with~(\ref{Eq:A}).
In particular, we have
\begin{equation}\label{eq:A1}
A_1(x) = x - \frac{N-1}2.
\end{equation}

\subsection{Contour integral}\label{sec:defcont}
We now replace the~$N$ sums in Equation~\eqref{eq:Z-Nsum} for the partition function by integrals over a contour $\mathcal{C}_D$ enclosing the non-negative integers less than or equal to~$D$. For that we use a function which has simple poles with residue~$1$ at all integers:
\begin{equation}\label{eq:f}
f(\xi) := \frac{\pi e^{-\I\pi\xi}}{\sin(\pi \xi)} = - \Gamma(\xi + 1)\Gamma(-\xi)e^{-\I\pi\xi}.
\end{equation}

Note that for any~$K$, only finitely many terms of the sum in~\eqref{eq:Z-Nsum} contribute to the coefficient of~$t^K$. Thus, when we want to compute any such coefficient, we can replace the sum by a finite one:
\begin{multline}
[t^K] Z(\mathbf{p}, g_s; t) = [t^K] Z_D(\mathbf{p}, g_s;t) := \\ \notag
[t^K] \frac{1}{N!}\frac{\Delta(\mathbf{R})}{\Delta(\mathbf{v})} \sum_{h_1, \ldots, h_N = 0}^D I(\mathbf{h},\mathbf{R})(\Delta(\mathbf{h}))^2\prod_{i=1}^N \frac{e^{g_s^r A_{r+1}(h_i)}(g_s/t)^{-A_1(h_i)}}{\Gamma(h_i + 1)},
\end{multline}
which is true as long as~$D \geq K + \frac{N-1}{2}$. 

\begin{remark}
While~$Z$ is a Laurent series that does not converge to a function, the truncated series $Z_D$ obviously does converge to a meromorphic function with domain~$\C$ for all the variables, since it is a finite sum of such functions. 
\end{remark}

Using the function~$f$ defined in equation~\eqref{eq:f} we can rewrite the function~$Z_D$ in terms of residues, if we restrict the domain of $g_s$ and~$t$ to $\C\backslash(-\infty, 0)$:
\begin{multline}\label{eq:Z_DsumRes}
Z_D(\mathbf{p}, g_s;t) = 
\frac{1}{N!}\frac{\Delta(\mathbf{R})}{\Delta(\mathbf{v})}  
\sum_{h_1, \ldots, h_N = 0}^D \Res_{z_1 \rightarrow h_1} \cdots \Res_{z_N \rightarrow h_N} \\
I(\mathbf{z},\mathbf{R})(\Delta(\mathbf{z}))^2\prod_{i=1}^N \frac{f(z_i) e^{g_s^r A_{r+1}(z_i)}(g_s/t)^{-A_1(z_i)}}{\Gamma(z_i + 1)}
\end{multline}
On the right-hand side, $(g_s/t)^{-A_1(z)}$ is defined as $\exp(-A_1(x) \log(g_s/t))$, which requires a choice of branch of the logarithm (one can see that the end result does not depend on this choice), and explains the change in domain. Here, it is important that $g_s$ and~$t$ are no longer just formal variables, but the arguments of a function.  

Finally, the sum over residues can be replaced by a contour integral:
\begin{align}\label{eq:Z_Dint}
Z_D(\mathbf{p},g_s;t) &= 
\frac{1}{N!}\frac{\Delta(\mathbf{R})}{\Delta(\mathbf{v})} \oint_{\mathcal{C}_D^N} \dd h_1 \cdots \dd h_N(\Delta(\mathbf{h}))^2 I(\mathbf{h},\mathbf{R}) 
\\ \notag & {\ } \qquad \qquad \qquad \qquad
\prod_{i=1}^N \frac{f(h_i) e^{g_s^r A_{r+1}(h_i)}(g_s/t
)^{-A_1(h_i)}}{\Gamma(h_i + 1)} .
\end{align}

\begin{remark}\label{Rem:ZvsZD}
Equation~\eqref{eq:Z_Dint} is an equality of functions, and the function on the left-hand side is defined as a (converging) series in~$t$, implying that the function on the right has the same series expansion at~$t=0$. Note that we have to work with~$Z_D$ because it is not possible to write a formula like~\eqref{eq:Z_DsumRes} for~$Z$, since it is only a formal series, and does not converge to any function. Another way to see this is that the integral over all non-negative integers does not converge, so that it is meaningless to take the coefficient of~$t^K$ in that integral.
\end{remark}


Rescaling the integration variables $h_i \rightarrow h_i/g_s$,  we get:
\begin{multline}
Z_D(\mathbf{p},g_s;t) = \frac{g_s^{-N^2}}{N!} \frac{\Delta(\mathbf{R})}{\Delta(\mathbf{v})} 
\oint_{\mathcal{C}_D^N} \dd h_1 \cdots \dd h_N(\Delta(\mathbf{h}))^2 \ \times \\ 
 I\left(\frac{\mathbf{h}}{g_s},\mathbf{R}\right)
\prod_{i=1}^N -\Gamma\left(-\frac{h_i}{g_s}\right) e^{g_s^r A_{r+1}(\frac{h_i}{g_s}) -\frac{I\pi h_i}{g_s}}\left(\frac{g_s}{t}\right)^{-A_1(\frac{h_i}{g_s})} .
\end{multline}

\subsection{Normal matrices and final formula}\label{sec:defmeasure}
As it is done in~\cite{BEMS}, we now replace the integration along the $N$ copies of the contour~$\mathcal{C}_D$ by integration over the space $\mathcal{H}_N(\mathcal{C}_D)$ of $N$ by~$N$ normal matrices with eigenvalues in~$\mathcal{C}_D$. We get
\begin{equation}
Z_D(\mathbf{p}, g_s; t)
) = \lim_{N \to \infty}\frac{1}{N!} \frac{\Delta(\mathbf{R})}{\Delta(\mathbf{v})} \int_{\mathcal{H_N(\mathcal{C}_D)}} \mathrm{d}M e^{- \Tr V(M) + \Tr (M\mathbf{R})} , 
\end{equation}
where~$V$ is as in formula\eqref{eq:potential}:
\begin{equation*}
V(\xi) = -g_s^r A_{r+1}(\frac{\xi}{g_s}) + g_s\log(\frac{g_s}{t}
)A_1(\frac{\xi}{g_s}) + \I\pi\xi -g_s\log(\Gamma(-\frac{\xi}{g_s})) + \I\pi g_s. 
\end{equation*}

In particular, it implies
\begin{equation}
Z(\mathbf{p}, g_s;t) \sim \frac{g_s^{-N^2}}{N!} \frac{\Delta(\mathbf{R})}{\Delta(\mathbf{v})} \int_{\mathcal{H_N(\mathcal{C}_D)}} \dd M e^{- \frac{1}{g_s} \Tr (V(M) - M\mathbf{R})},
\end{equation}
concluding the prove of Theorem~\ref{thm:Zmatrix}.

\section{Evidence for the $r$-BM conjecture} \label{Sec:Evidence}

Here, we present evidence for the $r$-BM conjecture. There are three independent reasons to believe the conjecture. 
The first is the computational evidence for the $r$-ELSV conjecture discussed in the introduction, together with the equivalence of that conjecture to the $r$-BM conjecture proved in Section~\ref{sec:SCHur}. The second is the general idea that the spectral curve for an enumerative problem should be given by its $(0,1)$-geometry, as discussed in~\cite{DMSS12}. The third reason is the generalization of the proof of the Bouchard-Mari\~{n}o conjecture in~\cite{BEMS}. 

In this section, we give a more in depth discussion of those last two reasons. 

\subsection{The $(g,n) = (0,1)$ geometry}
In~\cite{DMSS12}, the authors propose the idea that it should be a general aspect of the Eynard-Orantin topological recursion theory that the spectral curve for an enumerative problem should be given by the so-called $(0,1)$-geometry of the problem. In the case of Hurwitz problems, this means that the spectral curve should be given by their $1$-pointed generating function of genus~$0$. Evidence for this idea consists of several special cases where it works; in fact, in all known examples this seems to be true.

For completed Hurwitz numbers, the spectral curve suggested by the $(0,1)$-geometry was computed in~\cite{MSS}, and it coincides (up to an irrelevant change of variables) with the spectral curve of Conjecture~\ref{Conj:BM}. Furthermore, it was also shown in that paper that the quantization of that curve leads to an operator annihilating the principal specialization of the generating function for Hurwitz numbers, in accordance with another conjecture, formulated by Gukov and Su\l kowski in~\cite{GuSu}, providing further evidence for the $r$-BM conjecture.

\subsection{Spectral curve associated to a matrix model}\label{SSec:EvidenceMM}
In~\cite{BEMS}, a ``phys{\-}ics proof'' of the Bouchard-Mari\~no conjecture is given by first representing the generating function for Hurwitz numbers as a matrix model, and then showing that the spectral curve for this matrix model is equal to the one predicted by the Bouchard-Mari\~no conjecture. We generalized the first part of this proof in the previous section, where we showed that the generating function for completed Hurwitz numbers is given by the matrix model~\eqref{eq:Zmatrix}.


Unfortunately, it seems that the reasoning in~\cite{BEMS} does not constitute a precise mathematical proof, nor can it easily be made into one. On the other hand, all the reasoning in~\cite{BEMS} generalizes directly to the case of completed Hurwitz numbers. Thus, if a way could be found to transform this into a rigorous mathematical proof, it would immediately prove the $r$-BM conjecture.

\begin{remark}
There is another proof of the Bouchard-Mari\~no conjecture in~\cite{EMS09}, but it is based on the ELSV formula, so it is not useful for our purposes. 
\end{remark}

In the rest of this section, we briefly describe the steps taken in~\cite{BEMS} and how they generalize to completed Hurwitz numbers, and we note the places were we believe the reasoning is not mathematically rigorous. Since all the steps in~\cite{BEMS} generalize directly to our case, we do not repeat the detailed steps of that paper, and just give an overview of the reasoning. 

%
 
\subsubsection{Loop equations and topological expansion} It is a general theme in the theory of matrix models that they can be related to a spectral curve by way of so-called \emph{loop equations}. That is, to any matrix model one can associate a free energy $F=\log Z$ and a tower of $n$-point correlation functions $W_n(x_1,\dots,x_n)$, $n\geq 1$. Then, one can ask whether there exists a curve such that those invariants coincide with the symplectic invariants and $n$-point correlation forms associated to this curve by the Eynard-Orantin topological recursion procedure.

In general, the answer to this question is given by varying the integration variable in the matrix integral in a specific way. The resulting equations are called the loop equations, and in good situations they imply that the free energy and correlation functions of the matrix model are given by a specific spectral curve.  

To derive such loop equations, we need $F$ and~$W_n$ to have a Laurent series expansion in powers of~$g_s$ with finite tail (this is called the topological expansion property of the matrix model). In particular, this allows us to define invariants $F_g$ and~$W_{g,n}$ as the coefficients of powers of~$g_s$ in the expansions of $F$ and~$W_n$, which in turn makes it possible to compare them to the symplectic invariants and correlation forms of a spectral curve. In~\cite{BEMS}, it is shown that the matrix model discussed there has this topological expansion property, and the proof goes through in exactly the same way for our matrix model. 

\begin{remark}
In~\cite{BEMS}, the correlators~$W_{g,n}$ are themselves power series in~$g_s$. Because of the triangular nature of the relation between $W_{g,n}$ in that paper and the coefficients of powers of~$g_s$ in~$W_n$, it is immediate that those coefficients of powers of~$g_s$ are well-defined and they are non-zero only for finitely many negative powers of~$g_s$. This does mean that the powers of~$g_s$ in the expansion of~$W_n$ do not necessarily increase in steps of~$2$, but that does not present any problems in the rest of the reasoning.
\end{remark}

\subsubsection{Loop equations}
Given that our matrix model has the topological expansion property, the loop equations are derived as the invariance of the integral under a certain change of variables. In~\cite{BEMS}, the change
$$
M \rightarrow M + \epsilon \frac{1}{x-M}\frac{1}{y-R} 
$$
for~$\epsilon$ small is used, but this does not preserve the property of being a normal matrix, so we prefer the change
\begin{equation}\label{eq:change}
M \rightarrow M + \frac{\epsilon}{2} \frac{1}{x-M}\frac{1}{y-R} + \frac{\epsilon}{2} \frac{1}{y-R}\frac{1}{x-M} 
\end{equation}
which does preserve that property. In Appendix~D of~\cite{EO} the spectral curve equation and the Eynard-Orantin topological recursion for the correlators are derived from such a change of variables for a matrix model of the form~\eqref{eq:Zmatrix}. However, their derivation depends on the potential~$V(x)$ being a rational function of~$x$, independent of~$g_s$, neither of which holds for the matrix model in~\cite{BEMS} or the generalization described here. This does not affect their reasoning when deriving the spectral curve, but they really use those properties of~$V$ to derive the corresponding topological recursion for the~$W_{g,n}$. 

Furthermore, the invariance of the integral under the change of variables depends on the fact that the domain of integration does not change under this change of variables. When using the infinite contour of~\cite{BEMS}, this in fact holds, but for the finite contour~$\cC_D$ (see Remark~\ref{Rem:Zcontour}), it is not the case. That is, one easily sees that the change~\eqref{eq:change} does not affect the property of a matrix being Hermitian (normal matrix with real eigenvalues), but it sends the space of normal matrices with eigenvalues in~$\cC_D$ to the space of normal matrices with eigenvalues on some different contour~$\widetilde{\cC}_D$. 

\begin{remark}
If we were integrating over the space of diagonal matrices with eigenvalues in~$\cC_D$ instead of those that are diagonalizable using unitary matrices, the space would also effectively be invariant under the change of variables, since the integral would only depend on the homotopy type of the contour with respect to the non-negative integers. However, the unitary matrices spoil this symmetry.
\end{remark}

\subsubsection{Spectral curve for completed Hurwitz numbers}  Suppose that we would overcome the problems described above in some way. Then, the loop equations would lead to a spectral curve (depending on~$g_s$) and corresponding topological recursion for the~$W_{g,n}$. The proof of Conjecture~\ref{Conj:BM} could then be completed as in~\cite{BEMS}, using the relation between the $n$-point genus~$g$ correlation functions for Hurwitz numbers and the free energy of the matrix model
\begin{equation}\label{eq:ZandH}
\frac{\d^n H_{g,r}(R_1,\dots,R_n)}{\d R_1\cdots \d R_n} =\left. \frac{1}{g_s^n} \frac{\d^n F_g}{\d R_1 \cdots \d R_n } \right|_{g_s=0}\  
\end{equation}
and some properties of the topological recursion theory and its relation to matrix models. Together, those show that the spectral curve for completed Hurwitz numbers is given by $\cC_\mathrm{s}^{(r)}\colon x = -y^r + \log y$, concluding the evidence for the $r$-Bouchard-Mari\~no conjecture.

\end{document}